\documentclass[11pt]{amsart}

\usepackage{amsmath}
\usepackage{amssymb}
\usepackage{amsfonts}
\usepackage{amsthm}
\usepackage{enumerate}
\usepackage[mathscr]{eucal}

\usepackage{graphicx}
\usepackage{verbatim}

\newcommand{\F}{\mathcal F}

\newcommand{\Be}{\begin{equation}}
\newcommand{\Ee}{\end{equation}}

\newcommand{\Bm}{\begin{multline}}
\newcommand{\Em}{\end{multline}}
\newcommand{\Bea}{\begin{eqnarray}}
\newcommand{\Eea}{\end{eqnarray}}
\newcommand{\Beas}{\begin{eqnarray*}}
\newcommand{\Eeas}{\end{eqnarray*}}
\newcommand{\Benu}{\begin{enumerate}}
\newcommand{\Eenu}{\end{enumerate}}
\newcommand{\Bi}{\begin{itemize}}
\newcommand{\Ei}{\end{itemize}}

\def\bom{{\partial\Omega}}

\def\intslash{\rlap{\kern  .32em $\mspace {.5mu}\backslash$ }\int}
\def\qsl{{\rlap{\kern  .32em $\mspace {.5mu}\backslash$ }\int_{Q_x}}}

\def\emph#1{{\it #1 }}

\def\ga{\gamma}

\def\cf{{\it cf}}

\def\dist{{\text{\it dist}}}

\def\supp{{\text{\rm supp}}}

\def\inn#1#2{\langle#1,#2\rangle}

\def\noi{\noindent}

\def\meas{{\text{\rm meas}}}

\def\lc{\lesssim}

\def\eps{\varepsilon}

\def\la{\lambda}             
             
              \def\Si{\Sigma}

\def\om{\omega}

\def\fM{{\mathfrak {M}}}

\def\fR{{\mathfrak {R}}}

\def\bbM{{\mathbb {M}}}

\def\bbR{{\mathbb {R}}}

\def\bbZ{{\mathbb {Z}}}

\def\sA{{\mathscr {A}}}
\def\sM{{\mathscr {M}}}

\def\cA{{\mathcal {A}}}

\def\cC{{\mathcal {C}}}

\def\cE{{\mathcal {E}}}
\def\cF{{\mathcal {F}}}

\def\cH{{\mathcal {H}}}

\def\cM{{\mathcal {M}}}
\def\cN{{\mathcal {N}}}

 at 10 true pt

\def\be#1{\begin{equation}\label{ #1}}
\def\endeq{\end{equation}}
\def\endal{\end{align}}
\def\bas{\begin{align*}}
\def\eas{\end{align*}}
\def\bi{\begin{itemize}}
\def\ei{\end{itemize}}

\def\eps{\varepsilon}

\def\emph#1{{\it #1}}
\def\textbf#1{{\bf #1}}

\theoremstyle{plain}
   \newtheorem{theorem}{Theorem}[section]

   \newtheorem{proposition}[theorem]{Proposition}
   
   \newtheorem{lemma}[theorem]{Lemma}

   \newtheorem{theorem*}{Theorem}

\theoremstyle{remark}

\theoremstyle{definition}

\begin{document}

\title
[Problems  on averages  and  lacunary maximal functions]
{Problems  on averages \\ and  lacunary maximal functions}

\author[A. Seeger and J.  Wright]{Andreas Seeger and James  Wright}

\address{Department of  Mathematics\\ University of Wisconsin-Madison\\
480 Lincoln Drive\\ Madison, WI 53706, USA}
\email{seeger@math.wisc.edu}

\address{Maxwell Institute for Mathematical Sciences and the
School of Mathematics\\ University of
Edinburgh\\ JCMB, King's Buildings, Mayfield Road, \\ Edinburgh EH9 3JZ,
Scotland}

\email{J.R.Wright@ed.ac.uk}



\begin{thanks} {A.S. supported in part by 
NSF 0652890 and  MEC grant MTM2010-16518.
J.W. supported in part by an EPSRC grant.}
\end{thanks}

\begin{abstract}
We prove three results concerning convolution operators and lacunary maximal functions associated to dilates of measures.
First we obtain
an $H^1$ to $L^{1,\infty}$ bound for lacunary maximal operators under a
dimensional assumption on the underlying measure and an assumption on
 an  $L^p$ regularity bound for some $p>1$. Secondly,
we obtain a necessary and sufficient condition for  $L^2$ boundedness of
lacunary maximal operator associated to averages over convex curves in the
plane. Finally we prove an $L^p$ regularity result for such averages.
We formulate various open problems.
\end{abstract}

\maketitle

\section{Introduction}

We  consider a compactly supported finite Borel measure $\mu$ and define its
dyadic dilates by
$\inn{\mu_k}{f}= \inn{\mu}{f(2^k\cdot)}$. The main objects of this paper
are the convolutions $f\mapsto f*\mu_k$ and
 the   lacunary maximal function  given by
$$\sM f(x)= \sup_{k\in \bbZ} |f*\mu_k(x)|.$$
Throughout this paper the dilates  $2^k$ can be replaced by
more general lacunary dilates $\lambda_k$
satisfying $\inf_k \lambda_{k+1}/\lambda_k > 1$.

If
$\mu$  satisfies the condition
$\widehat \mu(\xi)=O(|\xi|^{-\eps})$ (some $\eps>0$),
then  $L^p$-boundedness of $\sM$ holds in the
range $1<p<\infty$, see {\it e.g.} \cite{duordf}.
For suitable classes of examples we discuss two problems,
namely what may happen in the limiting case $p=1$, and,  secondly, what
could be said  about boundedness  of $\sM$
if the above decay condition on $\widehat \mu$ is relaxed.

{\it Notation:} For two nonnegative quantities $A$ and $B$ let $A\lc B$ denote the statement that
$A\le  CB$ for some constant $C$. The Lebesgue measure of a set  $E$ is
denoted  by $\meas(E)$.  Throughout  we work with a fixed inhomogeneous dyadic frequency decomposition $\{P_k\}_{k=0}^\infty$. Let 
$\beta_\circ\in C^\infty_c(\bbR)$ be supported in $(-1,1)$ and 
equal to $1$ in $[-1/2,1/2]$. Define the operators $P_k$ by
$\widehat {P_0 f}(\xi)=\beta_\circ (|\xi|)\widehat f(\xi)$ if $k=0$ and
by
$\widehat {P_k f}(\xi)=(\beta_\circ (2^{-k}|\xi|)
-\beta_\circ (2^{-k+1}|\xi|))\widehat f(\xi)$ if $k>0$.

\subsection*{$\bold {H}^{\boldsymbol 1}\to \bold{L}^{\boldsymbol 1,\boldsymbol \infty}$ boundedness of lacunary maximal operators}

Concerning the case $p=1$, ${\mathscr M}$ can never be bounded on $L^1({\mathbb R}^d)$,
outwith the trivial case. One  can ask whether
the  smoothing condition
$\widehat \mu(\xi)=O(|\xi|^{-\eps})$ implies that  ${\mathscr M}$ is of weak-type $(1,1)$, i.e.
maps  $L^1({\mathbb R}^d)$ to the Lorentz space $L^{1,\infty}({\mathbb R}^d)$.
To the best of our knowledge no counterexample and no example
is known for the case that $\mu$ is a singular measure with the decay assumption on $\widehat \mu$.
%

For  various classes of singular measures  it has been observed that a
somewhat weaker endpoint inequality holds, namely that
$\sM$ maps the (usual isotropic) Hardy space $H^1({\mathbb R}^d)$ to $L^{1,\infty}({\mathbb R}^d)$. The first results in this direction are
due to Christ  \cite{christ} who used  some powerful variants  of  Calder\'on-Zygmund theory.
We formulate a theorem  which unifies  and extends some previous results,
with some simplification in the proofs. In what follows we let
 $\sA$ denote the convolution operator $\sA: f\mapsto f*\mu$.

\begin{theorem}\label{H1}
For $\rho>0$ let $\cN_\Si(\rho)$ be the minimal number of balls of radius 
$\rho$ needed to cover the support of $\mu$. Let  $0< s\le d$ and suppose that
\begin{align}
\label{covering}&\sup_{0<\rho<1} \rho^{d-s} \cN_\Si(\rho) <\infty
\\
\label{sobassu}
&\sup_{k>0} 2^{k(d-s)(1-\frac 1p)} \|P_k\sA\|_{L^p(\bbR^d)\to L^p(\bbR^d)} <\infty,\quad
 \text{ for some $p>1$.}
\end{align}
Then the  lacunary maximal operator ${\sM}$ maps $H^1({\mathbb R}^d)$
into $L^{1,\infty}({\mathbb R}^d)$.
\end{theorem}

The covering condition \eqref{covering}
is a dimensional assumption  on the support of the measure;
in particular when $s$ is an integer
 it is satisfied if $\supp (\mu)$ is contained in an
imbedded manifold of codimension $s$. The assumption
\eqref{sobassu} expresses an optimal smoothing result in the category of Besov spaces, for $p$ near $1$. We note that assumptions
\eqref{covering} together with assumption \eqref{sobassu}  imply stronger regularity results in Sobolev and Triebel-Lizorkin spaces, see \cite{pr-ro-se}.
When $p=2$  the  assumption \eqref{sobassu}
 is equivalent with the inequality
$\widehat \mu(\xi) = O(|\xi|^{-\frac{d-s}{2}})$. In this case we recover results
by Oberlin \cite{oberlin}  and  Heo \cite{heo1}, both extending
Christ's original result and method of proof in \cite{christ}.

{\it Examples.} (i) Our main new example concerns the case of arclength
measure on a compact  curve in $\bbR^3$  with nonvanishing curvature
and torsion.
Here  $d=3$, $s=2$ and the main regularity assumption
is a recent result of Pramanik and one of the authors \cite{pr-se},
established  for  $p< (p_W + 2)/p_W$ where  $p_W<\infty$ is an
  exponent for which  a deep inequality of T. Wolff \cite{wolff}
 on decompositions of cone multipliers holds.

(ii) As in   \cite{stw2}, \cite{heo2} one can consider  hypersurfaces
on which the Gaussian curvature does not vanish to infinite order. Then
of course \eqref{covering} holds with $s=1$ and, by
a result  of Sogge and Stein \cite{SoSt},  inequality \eqref{sobassu}
holds for some values of $p>1$ (in a range depending on the order of
vanishing of the curvature).

(iii) An interesting  example occurs in \cite{heo3} where  $\Sigma$ is a  portion of the light cone in $\Bbb R^d$ when $d\ge 5$. In this case one can see  condition
\eqref{sobassu} as a consequence of a sharp space time inequality
for spherical means, see inequality (10.9) in \cite{hns}.


\noi{\it Open problems.}  There are many,  concerning both the hypotheses and the conclusion of Theorem \ref{H1}.

(a) Given the dimensionality assumption \eqref{covering}, under what conditions does \eqref{sobassu} hold? More concretely, if $d\mu=\chi d\sigma$ and
$d\sigma$ is arclength measure on (a compact piece of) a curve of finite type in $\bbR^d$, does \eqref{sobassu} hold with $s=d-1$ and some $p>1$.
This is open in dimensions $d\ge 4$.  Similarly, is $\cM: H^1\to L^{1,\infty}$ for these examples?

(b) More generally, if $\Sigma$ is a manifold of finite type
(in the sense of ch. VIII, \S 3.2. of \cite{steinbook})
does it follow that the lacunary maximal operator
maps $H^1$ to $L^{1,\infty}$?
See \cite{taocyl} for a measure supported on a cylinder for which the
associated lacunary maximal  operator
which does map $H^1$ to $L^{1,\infty}$ but
for which the hypothesis \eqref{sobassu} does not hold.

(c)  In Theorem \ref{H1}, can one replace the usual isotropic dilations by nonisotropic ones, with the corresponding change of Hardy spaces, but keeping the
isotropic assumption \eqref{sobassu}?
See \cite{christ}  for results on averaging operators along curves in two
dimensions  and related results in \cite{stw2}, \cite{heo2}
for hypersurfaces. The
 question whether the
`maximal function along the $(t,t^2,t^3)$ maps the corresponding
anisotropic
Hardy space $\cH^1$  to $L^{1,\infty}$  is currently open.

(d) In  \cite{stw2} it was shown that the lacunary maximal operators
associated to hypersurfaces of finite type (with respect to an arbitrary dilation group) are of weak type $L\log\log L$. This  is the current result closest to a perhaps conjectured weak type $L^1$ bound. It is open whether one can prove a similar   result merely
 under the regularity assumption
\eqref{sobassu}.

{\it Remark.} We take this opportunity to mention a fallacious argument in the exceptional set estimate in \S5 of the article \cite{stw2}  and thank   Neal Bez for  pointing it out.
 A correction is posted on one of the authors'  website
(see the reference to \cite{stw2} below).

\subsection*{Lacunary maximal operators associated to convex curves in the plane}
Let $\Omega$ be a convex open domain in the plane with compact closure so that the origin is contained in $\Omega$.
We let  $\sigma$ be the arclength measure on the  boundary $\bom$ and consider the  question of $L^p$ boundedness of the lacunary maximal operator associated to $\bom$,
\begin{equation}\label{lacmaxdef}
 \cM f(x) =\sup_{k\in \bbZ} \Big|\int_{\bom}
 f(x- 2^{k} y) d\sigma(y) \Big|.
\end{equation}
As mentioned above,  if
$|\widehat \sigma(\xi)|=O(|\xi|^{-\eps})$ for some $\eps>0$, then $\cM$ is  $L^p$ bounded for $1<p<\infty$. Now  we
are aiming for much weaker hypotheses.

The decay of $\widehat \sigma$  is strongly related to a geometric
quantity. Given a unit vector $\theta$ let $\ell^+(\theta)$ be the unique
supporting line with $\theta$ an outer normal to $\bom$, i.e.
the affine line perpendicular to $\theta$ which intersects $\bom$
so that $\Omega$ is a subset of the halfspace $\{x: x=y-t\theta: t>0, y\in\ell^+(\theta)\}$. Similarly define
$\ell^-(\theta)$ as the unique
affine line perpendicular to $\theta$ which intersects $\bom$
and  $\Omega$ is a subset of the halfspace
$\{x: x=y+t\theta: t>0, y\in\ell^-(\theta)\}$.
For small $\delta>0$ define the arcs (or \lq caps\rq )
$$\cC^{\pm}(\theta,\delta) =\{y\in\bom: \dist(y, \ell^{\pm})\le \delta\}.$$
By a compactness consideration it is easy to see that there is $\delta_0>0$ so that for all $\theta\in S^1$ and all $\delta<\delta_0$ the  arcs
$\cC^{+}(\theta,\delta) $
 and $\cC^{-}(\theta,\delta)$ are disjoint.
Let,  $\Lambda(\theta,\delta)$ be the maximum of the length of these caps:
$$\Lambda(\theta,\delta) = \max_\pm \sigma(\cC^{\pm}(\theta,\delta))\,.$$
The analytic  significance of this quantity is that it gives  a
 very good estimate   for the size of Fourier transform $\widehat\sigma $,
namely
for every $\theta\in S^1$ and $R\ge 1$
\begin{equation} \label{FTestimate}
|\widehat \sigma(R\theta)| \le C_\Omega \Lambda(\theta, R^{-1}).
\end{equation}
This is shown in \cite{bnw} under the hypothesis for convex domains in the plane with smooth boundary, with no quantitative assumption on the second derivative. The general case follows by a simple approximation procedure (see also  \cite{pod1}, \cite{pod2} for similar observations).


We note that for classes of multipliers  satisfying standard symbol assumptions one can insure boundedness under rather weak decay
assumptions on the symbol, but
it is not clear what the optimal conditions are;
 moreover the usual  method of square-functions is often not the appropriate
tool (\cf. \cite{cghs}, \cite{ghs}).
In this light the  following characterization  for $p=2$ in terms of the quantities $\Lambda(\theta,\delta)$
  is perhaps surprising (as well as the fact that it can be proved using simple square-function arguments).

\begin{theorem}\label{convex}
The operator $\cM$ is bounded on $L^2(\bbR^2)$ if and only if
\begin{equation}\label{ell2condcaps}
\sup_{\theta\in S^1} \int_0^{\delta_0} \Lambda (\theta,\delta) ^2 \frac{d\delta}{\delta}
\ < \ \infty.
\end{equation}
\end{theorem}

{\it Problem:} For $q\neq 2$, find necessary and sufficient conditions  on $\bom$ in order for $\cM$ to be bounded on  $L^q(\bbR^2)$.

It may be interesting to look at specific \lq flat\rq \  examples.
By testing $\cM$ on functions supported in thin strips we
shall obtain a  necessary condition
\begin{equation}\label{Lqnec}
\sup_{\|f\|_q=1}\|\cM f\|_{q} \ge c
\sup_{\theta\in S^1}  \Big(\int_0^{\delta_0} \Lambda (\theta,\delta) ^q \frac{d\delta}{\delta}\Big)^{1/q}
\end{equation}
for all $q$.
We note  that if ${\bom}$ has only one ``flat'' point near which ${\bom}$
can be parametrized as the graph of $C+\exp(-1/|t|^a)$ with $C\neq 0$,
then this $L^q$ condition
holds iff $a<q$. Thus in this case  $L^2$ boundedness of $\cM$
 holds if and only if $a<2$.

\subsection*{$\bold L^{\bold p}$-regularity of averages}

As in  the previous  section we consider convex curves  $\Sigma$, say boundaries of a convex domain but the position of the origin will not play a role now.
The estimate \eqref{FTestimate} can be interpreted as an $L^2$ regularity result for the integral  operator
$\cA: f\mapsto f*\sigma$.
We reformulate this with the standard dyadic frequency decomposition $\{P_k\}_{k=0}^\infty$ as
above. Then setting
\begin{equation}\label{omk}
\om_k=\frac{1}{\sup_{\theta\in S^1}\Lambda(\theta, 2^{-k})}
\end{equation}
\eqref{FTestimate} says that the inequality
\begin{equation}\label{LpSob}
\Big\| \Big(\sum_{k>0} \omega_k^2 |P_k \cA f|^2\Big)^{1/2}\Big\|_p \lc \|f\|_p
\end{equation}
holds for $p=2$.

We are now interested  in analogous $L^p$ regularity results, {\it i.e.} we wish to determine the  range of $p$ for which \eqref{LpSob} holds,
with the optimal weight $\om_k$ in  \eqref{omk}.
In case the curvature vanishes somewhere  one expects this inequality to hold for some
$p\neq 2$; for example if $\Sigma$ is of finite type, and if $m$ is the order of maximal contact of tangent lines with $\Sigma$ then \eqref{LpSob} holds with the optimal $\omega_k= 2^{-k/m}$ for the  range $\frac{m}{m-1}<p<m$, see \cite{StW}, \cite{christ2} (and also \cite{se-duke}, \cite{yang} for variable coefficient analogues).
Given these known examples we are mainly interested in very flat cases.
We shall prove \eqref{LpSob} for a family of  curves with additional
hypotheses which cover interesting examples for which the curvature
 vanishes of infinite order at a point. In those flat cases one gets  \eqref{LpSob} in the full range  $1<p<\infty$.

After a localization
we assume that (part of) the convex curve  is given as a graph
$(t, \gamma(t))$ for $0\le t\le 1$ and consider the integral operator
$$A f(x)=\int_0^1 f(x_1-t, x_2-\gamma(t)) dt.$$

\begin{theorem} \label{Lpsobthm}
Let  $\gamma$ be of class $C^3$ on $(0,1]$ and of class $C^2$ on $[0,1]$.
Assume that
$\gamma(0)=\gamma'(0)=\gamma''(0)=0$,
 $\gamma''$ is nonnegative  and strictly increasing on $[0,1]$, and furthermore,
the limit
$$b=\lim_{t\to 0+} \frac{\gamma''(t)}{t\gamma'''(t)}
$$
exists with $b\in [0,\infty)$.
Let $$w_k=\frac{1}{\gamma^{-1}(2^{-k})}.$$ Then
the inequality
\begin{equation}\label{LpSobR}
\Big\| \Big(\sum_{k>0} w_k^2 |P_k A f|^2\Big)^{1/2}\Big\|_p \lc \|f\|_p
\end{equation}
holds for
$1+ \frac{b}{1+b}<p<2+b^{-1}$ (and thus for $1<p<\infty$ if $b=0$).
\end{theorem}

\noi{\it Examples.} We note that,  in the setup  of  this theorem,    $w_k \approx \om_k$ as defined in \eqref{omk}.
In the  case  $\gamma_\circ(t)= t^m$ with $m>2$ we recover the
 known result mentioned above  since   $b^{-1}=m-2$ and
$w_k\approx 2^{-k/m}$. In the flat case $\gamma_1(t)= e^{-t^{-a}}$ we have
$b=0$ and $w_k\approx k^{-1/a}$. We may consider evem flatter cases: Let $\exp_*^n$ the $n$-fold iteration $\exp\circ\dots\circ \exp$ and $\log_*^n $ the $n$-fold iteration of $\ln$. For large $C>0$ consider
 $\gamma_2(t)= \exp(- \exp_*^n(Ct^{-\la}))$. Then
 $b=0$ and $w_k\approx (\log_*^n (e^n+k))^{-1/\la}$.


\noi{\it Open problems.}

(i) For the curves $\Gamma(t)=(t,\gamma(t))$ featured in Theorem \ref{Lpsobthm}
let $\fM$ be the maximal function along $\Gamma$
(as in \eqref{maxalong} below).
Does $\fM$ map the
Hardy-space $H^1_{\text{prod}}$
(associated with  the two-parameter  dilations $(t_1\cdot, t_2\cdot)$)
to  the Lorentz space $L^{1,2}$?
Similar questions can be formulated for certain singular integrals along
$\Gamma$. For the finite type case ($\gamma(t)=t^m$) such estimates can be found in \cite{setao}. For the flat cases one would need to further
 explore Hardy space structures associated to the curve $\Gamma$.

(ii) Let $\mathscr D^{1/m}f= \F^{-1}[(1+|\xi|^2)^{1/2m}\widehat f]$, the fractional Bessel derivative of order $1/m$. For the finite type $m$  case there
are endpoint estimates involving Lorentz-spaces $L^{p,2}$, namely
it was shown in \cite{setao} that
$\mathscr D^{1/m} A$ maps    $L^{m,2}$ to $L^m$ and $L^{m',2}$ to $L^{m'}$, see  \cite{christ2} for the sharpness of such results.
 It would  be   interesting to   investigate
sharp regularity results for general $\gamma$. One  aims to bound
the square-function
$(\sum_{k>0} w_k^2 |P_k A f|^2)^{1/2}$ in
natural Orlicz or Orlicz-Lorentz
spaces  associated with $\gamma$.
As suggested by the method in \cite{setao}, endpoint results  should be related to a resolution of
problem (i).

(iii) Can one prove \eqref{LpSob} for more general convex domains; in particular, can one relax the  monotonicity assumption  on $\gamma''$ in
Theorem  \ref{Lpsobthm}?

\medskip

{\it This paper.} The proof of  Theorem \ref{H1}  will be
given in \S\ref{H1proof}.
The proof of Theorem \ref{convex} is in \S\ref{convexproof}.
 The proof of Theorem \ref{Lpsobthm} is in
in \S\ref{Lpsobproof}.
In \S\ref{hyp} we formulate yet another problem concerning
lacunary maximal functions for dilates of a simple
Marcinkiewicz multiplier.

{\bf Acknowledgements}: We would like to thank
Terry Tao for conversations  on the subject of lacunary maximal operators
and collaboration on  \cite{stw1} and \cite{stw2}.

\section{Proof of theorem \ref{H1}}\label{H1proof}

{\it Atomic decompositions.}
In what follows $\psi$
will denote  a nontrivial $C^\infty$ function  with compact support in $\{x: |x|\le 1/2\}$ so that $\widehat\psi(\xi)\neq 0$ for $\tfrac 14\le|\xi|
\le 4$ and
$\widehat \psi$ vanishes to order $10 d$ at the origin. We let
$\psi_l= 2^{-ld}\psi(2^{-l}\cdot)$.

We use the atomic decomposition  based on a square-function characterization
as given in \cite{crf}; for variants and applications of this method see
\cite{se-stud}, \cite{stw1}, \cite{hns}.
Given $f$ in $H^1$ one can write $f=\sum_Q b_Q$ where this sum
converges in $H^1$, each $b_Q$ is supported in the double of the
dyadic cube $Q$ with sidelength $2^{L(Q)}$, and has the
following fine structure.
We have $$b_Q =\sum_{j\le L(Q)}\psi_j*\psi_{j}*b_{Q,j}$$ where $b_{Q,j}$
can be decomposed  as
$b_{Q,j}= \sum_{R\in \fR(Q,j)} e_R$, the families $\fR(Q,j)$  consist
of dyadic cubes $R$ of sidelength $2^j$ contained in $Q$ with disjoint interior, the bounded function $e_R$ is  supported on $R$, and finally
\begin{equation}
\sum_Q |Q|^{1/2}  \Big(\sum_{j\le L(Q)}\sum_{R\in \fR(Q,j)}\|e_R\|_2^2\Big)^{1/2} \lc \|f\|_{H_1}.
\end{equation}
We set
$$
\gamma_{Q,j}=\Big(\sum_{R\in \fR(Q,j)}\|e_R\|_2^2\Big)^{1/2}, \qquad
\gamma_{Q}=\Big(\sum_{j\le L(Q)} \gamma_{Q,j}^2\Big)^{1/2},
$$
and note that
\begin{align*}
\|b_{Q,j}\|_2&\lc \ga_{Q,j}\, ,
\\
\|b_{Q}\|_2&\lc \ga_Q\, ,
\end{align*} and
$$\sum_Q|Q|^{1/2}\ga_Q \lc \|f\|_{H^1}\,.$$


\medskip

{\it The weak type inequality.} The condition in \eqref{sobassu}
 becomes more restrictive as $p$ increases and therefore we may assume $p\le 2$.

By a scaling argument we may assume that
$\mu_0=\mu$ is supported in the unit ball  centered at the origin.
(the operator norm will depend on that scaling).

We need to show that
\Be\label{weaktype}
\meas \big(\big\{x: \sup_k \big| \mu_k* \sum_Q b_Q\big|
 >\alpha\big\}\big) \lc \alpha^{-1}\|f\|_{H^1}\,.
\Ee
To achieve this we assign for each $Q$ an  integer  $\tau(Q)$ depending on $\alpha$,
defined as follows. We first let  $\widetilde \tau(Q)$ be  the smallest
integer $\tau$  for which
$$ 2^{(d-s)\tau } 2^{sL(Q)} \ge \alpha^{-1}|Q|^{1/2}\gamma_Q\,$$
(or $-\infty$ if there is no such smallest integer)
and define $$\tau(Q)=\max \{L(Q), \widetilde \tau(Q)\}\, .$$

We form an exceptional set depending on $\alpha$ by  $$\cE=
\bigcup_Q \bigcup_{k: L(Q)<k\le \tau(Q)} (\supp (\mu_k)+ Q^*) $$
where  $Q^*$ is the tenfold dilate of $Q$ with respect to its center.
Note that
$(\supp \mu_k+ Q^*)$ is contained in the $2^k$-dilate of a $C2^{L(Q)-k}$ neighborhood of $\supp (\mu)$. By the assumption on $\mu$ this
neighborhood can be covered with $\lc 2^{(k-L(Q))(d-s)}$
balls of radius $2^{L(Q)-k}$.
Thus
$$\meas(\supp (\mu_k)+ Q^*) \le C_1 2^{kd} 2^{(k-L(Q))(d-s)} 2^{(L(Q)-k)d} =
C_12^{k(d-s)} 2^{L(Q)s}$$
and thus
$$
\meas(\cE)\lc
\sum_{Q:L(Q)< \tau(Q)}
2^{\tau(Q)(d-s)} 2^{L(Q)s}\,.
$$
By the minimality property
of $\tau(Q)$ in its definition  it follows that for the case
 $L(Q)<\tau(Q)$ the inequality
$2^{(\tau(Q)-1)(d-s)}2^{L(Q)s}\le \alpha^{-1}|Q|^{1/2}\gamma_Q$ is sa\-tis\-fied.
Thus
\Be\label{excsetest}
\meas(\cE)\lc
\sum_{Q:L(Q)< \tau(Q)}   \alpha^{-1}  |Q|^{1/2}\gamma_Q \lc  \alpha^{-1}\|f\|_{H^1}.\Ee

We split $\sup_k |f*\mu_k(x)|$ into three parts depending on $\alpha$.
\begin{align*}
I(x)&=\sup_k\Big|\mu_k*\sum_{Q: \tau(Q)<k} b_Q\Big|\,,
\\
II(x)&=\sup_k\Big|\mu_k*\sum_{Q: L(Q)<k\le \tau(Q)} b_Q\Big|\,,
\\
III(x)&=\sup_k\Big|\mu_k*\sum_{Q: k\le L(Q)} b_Q\Big|\,.
\end{align*}
Note that $II$ is supported in $\cE$
and thus
\Be
\label{reductionthreeparts}
\meas \big(\big\{x: \sup_k \big| \mu_k* \sum_Q b_Q\big| >\alpha\big\}\big) \lc\\
\frac{\|I\|_p^p}{\alpha^p}+ \meas(\cE)+
\frac{\|III\|_1}{\alpha}
\Ee
where $p$ is as in assumption \eqref{sobassu}.

The estimation for $\|III\|_1$ is straightforward and just uses the $L^2$ boundedness of the lacunary maximal operator.
Note that for $k\le L(Q)$ the function $\mu_k*f$ is supported in $Q^*$. We estimate
\begin{align}
\|III\|_1&\le \sum_Q \big\|\sup_{k\le L(Q)} |\mu_k* b_Q|\big \|_1\notag
\\
& \le \sum_Q |Q^*|^{1/2} \|\sM b_Q\|_2\notag
\\ & \lc\sum_Q |Q|^{1/2} \|b_Q\|_2 \lc \|f\|_{H^1}\,.
\label{IIIest}
\end{align}

We turn to the main term $I$ and estimate
\begin{align}
\|I\|_p^p&= \Big\| \sup_{k}\big| \mu_k* \sum_L
\sum_{\substack {Q:L(Q)=L\\ \tau(Q)\le k}}
\sum_{n\ge 0}  \psi_{L-n}*\psi_{L-n} *b_{Q,L-n}\big|\,\Big\|_p
\notag
\\
&\le \Big(\sum_{n=0} \Big(\sum_k \big\|I_{n,k}
\big\|_p^p\Big)^{1/p}
\Big)^p
\label{Mink}
\end{align}
where
$$I_{n,k}=
\mu_k* \sum_L
\sum_{\substack {Q:L(Q)=L\\ \tau(Q)\le k}}
 \psi_{L-n}*\psi_{L-n} *b_{Q,L-n}.$$
 We shall bound $\sum_k\|I_{n,k}\|_p^p$ with some exponential gain in $n$.

We first observe as a consequence of assumption \eqref{sobassu}
the simple convolution inequality
$\|\mu*\psi_j* g\|_p \lc 2^{-j(d-s)/p'}\|g\|_p$. Indeed this follows from \eqref{sobassu} by observing that $\|P_k \psi_j\|_1\lc 2^{-|j-k|2d}$ where we use the vanishing moment assumption on $\psi_j$.
By scaling, the operator of convolution with $\mu_k*\psi_{L-n}$ has the same operator norm as the operator of convolution with $\mu_0*\psi_{L-n-k}$  which is
$O(2^{- (k+n-L)(d-s)/p'})$. Because of the almost orthogonality of the $\psi_j$ we may apply the inequality
$\|\sum\psi_j*g_j\|_p\lc (\sum_j\|g_j\|_p^p)^{1/p}$.
Thus  we get
\begin{align*}
\|I_{n,k}\|_p &\lc \Big(\sum_L\Big\|
\mu_k* \psi_{L-n}*
\sum_{\substack {Q:L(Q)=L\\ \tau(Q)\le k}}
 b_{Q,L-n}\Big\|_p^p\Big)^{1/p}
\\&\lc 2^{- (k+n-L)(d-s)/p'}
\Big(\sum_L\Big\|
\sum_{\substack {Q:L(Q)=L\\ \tau(Q)\le k}}
 b_{Q,L-n}\Big\|_p^p\Big)^{1/p}
\\&\lc 2^{- (k+n)(d-s)/p'}
\Big(\sum_{\substack {Q: \tau(Q)\le k}}  2^{L(Q)(d-s)(p-1)} \,
\|b_{Q,L-n}\|_p^p\Big)^{1/p}\, ,
\end{align*}
where the last inequality follows from the  disjointness of the $Q$ with fixed $L=L(Q)$.
Thus  we get after interchanging summations and summing in $k\ge \tau(Q)$
\begin{align*}
\Big(\sum_k\|I_{n,k}\|_p^p\Big)^{1/p} &\lc
2^{-n(d-s)/p'}\Big(\sum_Q
2^{(L(Q)-\tau(Q))(d-s)(p-1)}
\|b_{Q,L-n}\|_p^p\Big)^{1/p}
\\
&\lc
2^{-n(d-s)/p'}\Big(\sum_Q
2^{(L(Q)-\tau(Q))(d-s)(p-1)} |Q|^{1-p/2}
\gamma_Q^p
\Big)^{1/p}\,.
\end{align*}
In the last estimate  we have used  H\"older's inequality and
$\gamma_{Q,j}\le \gamma_Q$.
By the definition of $\tau(Q)$ we have
$$2^{-(d-s)\tau(Q)} \le  2^{-sL(Q)} \alpha|Q|^{-1/2}\gamma_Q^{-1},$$
no matter whether $\tau(Q)>L(Q)$ or $\tau(Q)=L(Q)$.
This leads to
$$
2^{(L(Q)-\tau(Q))(d-s)(p-1)} |Q|^{1-p/2}
\gamma_Q^p
\le \alpha^{p-1}|Q|^{1/2}\gamma_Q
$$
and consequently we get
$$\sum_k\|I_{n,k}\|_p^p\le C^p 2^{-n(d-s)(p-1)} \alpha^{p-1}
\sum_Q|Q|^{1/2}\gamma_Q
$$
which by \eqref{Mink} implies
\Be\label{Iest}\|I\|_p^p\le \widetilde C^p  \alpha^{p-1}
\sum_Q|Q|^{1/2}\gamma_Q.\Ee
We finish the proof of \eqref{weaktype} by combining
\eqref{reductionthreeparts}, \eqref{excsetest}, \eqref{IIIest} and \eqref{Iest}.
\qed

\section{Proof of Theorem \ref{convex}}\label{convexproof}



We shall first prove the necessary condition \eqref{Lqnec} for
$L^q$-boundedness of $\cM$. We check the lower bound by providing an example \
for $\theta=e_2$ and the
 general case follows by rotating the curve. From the positivity of $\cM$ and translation invariance we may reduce to the case where $\cM$ is replaced by the
maximal operator
\begin{equation} \label{maximal-form}
\begin{aligned}
 &Mf (x)=  \sup_{k\in {\mathbb Z}} |A_k f(x)|
\\ &\text{ with } A_kf(x):=
\Bigl|\int_{|t|\le \epsilon} f(x_1 - 2^k t, x_2 - 2^k L + 2^k \gamma(t))  dt
\Bigr| ;
\end{aligned}
\end{equation}
here $t\mapsto \gamma(t)$ is a convex function with $\gamma(0)=0$, $L>0$, $\eps>0$ and the line $\{x_2=L\}$ is a supporting line at $(0,L)$.

We test $M$ on the functions
$$f_\eta(x)= (4\eta)^{-1/q}\chi_{E_\eta}(x)\text{ where }
E_\eta  :=  \{x\in {\mathbb R}^2 : \, |x_1| \le 1, \ |x_2| \le \eta  \}
$$
for  small $\eta>0$. Then  $\|f_\eta\|_q=1$.  For each $k\in {\mathbb Z}$ set
$$
F_{\eta,k} \ = \ \{x\in {\mathbb R}^2 : \,|x_1|\le 1/4, \ 2^k L \le x_2 \le 2^k L + \eta/4 \}.
$$
We define $k_0$ to be the smallest integer $k$ satisfying
$2^k  \ge  \eta  \max\bigl(1/4L, 1/\gamma(\epsilon)\bigr)$ and note that when $k\ge k_0$,
the sets $F_{\eta,k}$ are disjoint and we have the lower
bound $$Mf_\eta(x) \ge (2\eta)^{-1/q}\gamma^{-1}(2^{-k}\eta) \,\text{ for } x\in F_{\eta,k}.$$
Therefore
$$(4\eta)^{-1} \sum_{k\ge k_0} |F_{\eta,k}
| [\gamma^{-1}(2^{-k}\eta)]^q \ \le \ \sum_{k\ge k_0} \int_{F_{\eta,k}} Mf_\eta(x)^q dx \
\le \ \|Mf_\eta\|^q_{q}.
$$
Since $|F_{\eta,k}| = \eta/4$ for each $k$ and $\sigma({\mathcal C}(e_2,2^{-k}\eta) \le B \gamma^{-1}
(2^{-k}\eta)$ for some $B>0$ depending only on $\gamma$, we have
\begin{align*}
B^{-q} \int_0^{\delta_1} \sigma({\mathcal C}(e_2,\delta))^q \frac{d\delta}{\delta} \ &\le B^{-q} \sum_{k\ge k_0} \int_{2^{-k-1}\eta}^{2^{-k}\eta} \sigma({\mathcal C}(e_2,\delta))^q
\frac{d\delta}{\delta} \\& \le  \sum_{k\ge k_0} 4\eta^{-1}|F_{\eta,k}| [\gamma^{-1}(2^{-k}\eta)]^q
\end{align*}
for some $\delta_1>0$ depending  only on $\gamma$.
Thus from the two previous chains of inequalities we obtain
$\int_0^{\delta_1} \sigma({\mathcal C}(e_2,\delta))^q \frac{d\delta}{\delta}
\le 8 \|Mf_\eta\|_q^q$ and this
completes the proof of \eqref{Lqnec}.

We now assume $q=2$ and  that the condition on
the caps ${\mathcal C}(\theta,\delta)$ in \eqref{ell2condcaps} is satisfied.
By a  partition of unity,  the translation invariance and positivity
of ${\cM}$,
we may suppose that the maximal operator  is of the  form
\eqref{maximal-form}.
The averaging operators $A_k$ are convolution  operators with Fourier
multipliers
$$
m_k(\xi) \ = \ m_0 (2^k \xi)  \ := \ \int_{|t|\le \epsilon} e^{-  i 2^k
[ \xi_1 t +  L \xi_2 -  \xi_2 \gamma(t) ]} \, dt.
$$
For small $\xi$ we have the trivial bound
\begin{equation}\label{easy-bound}
|m_k (\xi) - 2\epsilon| \ \le \ C 2^k |\xi|
\end{equation}
where $C$ is a universal constant.
For large $|\xi|$ we will use the bound
\begin{equation}\label{P}
|m_k(\xi)| \ \le \ C \Lambda( \tfrac{\xi}{|\xi|}, (2^k|\xi|)^{-1})
\end{equation} which follows from \eqref{FTestimate}.

Now fix a Schwartz function $\Phi$ with $\int \Phi(x)\,dx = 2 \epsilon$
and define $\Phi_k (x) := 2^{-2k} \Phi(2^{-k}x)$. In order
to prove the $L^2$ boundedness of $M$, we note the
pointwise bound
$$
M f (x) \ \le \ \sup_{k\in {\mathbb Z}} |\Phi_k * f(x)| \ + \
\Big(\sum_{k \in {\mathbb Z}} | A_k f(x) - \Phi_k * f(x) |^2\Big)^{1/2}.
$$
The first
term on the right hand side is dominated by the  Hardy-Littlewood maximal function of $f$ and thus defines a bounded operator on all $L^q$, $q>1$.
Therefore  it suffices
by
Plancherel's theorem for the second term  to show that the
function
$$
\xi \ \to \ \sum_{k\in {\mathbb Z}} |m_0(2^k \xi)- \widehat \Phi(2^k \xi) |^2
$$
is a bounded function in $\xi$.
From \eqref{easy-bound} and \eqref{P}, we see that the boundedness
of this function of $\xi$ will follow if we can show that
$$
I(\xi) \ := \ \sum_{k: 2^k |\xi| \ge C} \sigma\bigl({\mathcal C}
(\tfrac{\xi}{|\xi|},
[2^k |\xi|]^{-1})\bigr)^2
$$
is uniformly bounded in $\xi$ for some large $C$. Now
$$
\sigma\bigl({\mathcal C}(\tfrac{\xi}{|\xi|},
[2^k |\xi|]^{-1})\bigr)^2 \ \le \ln 2 \int_{[2^k |\xi|]^{-1}}^{[2^{k-1}|\xi|]^{-1}}
\sigma\bigl({\mathcal C}(\tfrac{\xi}{|\xi|},\delta)\bigr)^2 \, \frac{d\delta}{\delta}
$$
and so
$$
I(\xi) \lc \sup_{\theta\in S^1} \int_0^{\delta_0} \sigma\bigl({\mathcal C}
(\theta, \delta)\bigr)^2 \frac{d\delta}{\delta},
$$
establishing the sufficiency part of Theorem \ref{convex}.
\qed

\section{Proof of Theorem \ref{Lpsobthm}}\label{Lpsobproof}
We let $k_\circ=\min\{k: 2^{-k}\le \tfrac 14\gamma''(1)\}$
and only need to consider the terms $P_k A f$ with $k>k_\circ$.
Define
$$h(t)= t^2 \gamma''(t)$$ so that $\gamma(t)\le h(t)$.
For $k>k_\circ$ we  define a finite
increasing  sequence $\{t_{k,n}\}_{n=1}^{N_k}$ so that $\gamma''$ doubles on those points (as long as $t_{k,n}<1$) and denote
the corresponding images under $\gamma''$ by $\rho_{k,n}$. We set
\begin{align}  \label{tk0}
t_{k,0}&= h^{-1}(2^{-k})
\\ \label{rhok0}
\rho_{k,0}&= \gamma''(t_{k,0}),
\end{align}
and, for $n\ge 1$, set
\Be\label{rhokn}
\rho_{k,n}=\begin{cases}  2^n\gamma''(t_{k,0})
& \text{ if } \gamma''(t_{k,0})\le  2^{-n-1} \gamma''(1)
\\
\gamma''(1) & \text{ if } \gamma''(t_{k,0})>2^{-n-1}\gamma''(1)
\end{cases}\, ,
\Ee
and
\begin{equation}\label{tkn}
t_{k,n}={\gamma''}^{-1} (\rho_{k,n})
\,.
\end{equation}
Define
\begin{align*}
A_{k,0} f(x) &=\int_0^{t_{k,0}} P_k f(x_1-t, x_2-\gamma(t)) \,dt
\\
A_{k,n} f(x) &=\int_{t_{k,n-1}}^{t_{k,n}} P_k f(x_1-t,x_2-\gamma(t)) \,dt, \quad n\ge 1.
\end{align*}
If we let $$N_k=1+\max\{\nu: \gamma''(t_{k,\nu})\le \frac 12 \gamma''(1)\}$$
then $t_{k,n}<1$ if $n\le N_k-1$ and $t_{k,n}=1$ for $n\ge N_k$; consequently
$A_{k,n}=0$ for $n>N_k$. By Minkowski's inequality we have
$$\Big\|\Big(\sum_{k>k_\circ} |w_k P_k A f|^2\Big)^{1/2}\Big\|_p
\le \sum_{n=0}^\infty
\Big\|\Big(\sum_{\substack{k:k>k_\circ\\n\le N_k}}
 |w_k  A_{k,n}f|^2\Big)^{1/2}\Big\|_p\,.
$$

Theorem \ref{Lpsobthm}
follows from the following two propositions by  interpolation; this can be seen as a variant of  arguments in \cite{christ2}, \cite{se-duke},
\cite{yang}.

\begin{proposition}\label{AknL2} For $n\ge 0$
 $$ \Big\|\Big(\sum_{k\ge k_\circ} |w_k  A_{k,n}f|^2\Big)^{1/2}\Big\|_2 \le C 2^{-n/2} \|f\|_2.$$
\end{proposition}
\begin{proposition}\label{AknLp}
 For any $ \epsilon>0$ and for $1<p<\infty$,
$$ \Big\|\Big(\sum_{k\ge k\circ} |w_k  A_{k,n}f|^2\Big)^{1/2}\Big\|_p \le C_\epsilon2^{n( \epsilon+b)} \|f\|_p
$$
for all $n\ge 0$.
\end{proposition}

For the proof of the propositions we  need to relate the numbers $t_{k,0}$ to
 the weights $w_k$, and we get
\Be \label{wk-tk0}  w_k t_{k,0} \le 1 \quad\text{ for } k>k_\circ.
\Ee
For  $k>k_\circ$ the range of $h$ includes $2^{-k}$, and thus  \eqref {wk-tk0}
follows from part (i) of the following lemma. Part (ii) shows that
\eqref{wk-tk0}
is effective.
\begin{lemma}\label{hinvest} Let $\gamma$ be as in Theorem \ref{Lpsobthm}.

(i) If  $0\le s\le \gamma(1)$ then $s\le h(1)$ and
$h^{-1}(s)\le \gamma^{-1}(s)$.

(ii) If $0\le s\le h(1/3)$ then  $s\le \gamma(1)$  and
$\gamma^{-1}(s)\le 3h^{-1}(s)$.
\end{lemma}
\begin{proof}
Note that the range of $h$ and $\gamma''$ on $[0,1]$ is the same.
(i) follows from
$\gamma(t)\le h(t)= t^2\gamma''(t)$
 for $0\le t\le 1$.
But since  $\gamma(0)=\gamma'(0)=0$ and since $\gamma''$ is increasing we
get the better  inequality   $\gamma(t)\le h(t)/2$
which is immediate from  Taylor's theorem.
(ii) follows from  $h(t)\le \gamma(3t)$ for $0\le t\le 1/3$ which
holds by the monotonicity of $\gamma'$ and $\gamma''$; indeed
$t^2\gamma''(t)\le t\int_t^{2t}\gamma''(u) \,du \le t\gamma'(2t)\le
\int_{2t}^{3t}\gamma'(u) du\le \gamma(3t).$
 \end{proof}

We will now turn to the proof of the propositions.
Proposition \ref{AknL2}
relies on van der Corput's lemma (ch.VIII, \S1.2. in\cite{steinbook}),
while the proof of Proposition  \ref{AknLp} relies on the eight authors' theorem \cite{eight} on the boundedness of the maximal operator along $(t,\gamma(t))$ under the $\gamma'$ doubling hypothesis.

\begin{proof}[Proof of Proposition \ref{AknL2}]
We may assume $n\le N_k$. Then for $n\ge 1$ we need to estimate the multiplier
$$m_{k,n}(\xi)= \int_{t_{k,n-1}}^{t_{k,n}} \exp(i(\xi_1 t+\xi_2 \gamma(t))) dt$$
when $2^{k-1}\le |\xi|\le 2^{k+1}$.

For $t\ge t_{k,n-1}$ we have
$$ \gamma''(t)\ge \gamma''(t_{k,n-1})= 2^{n-1} \gamma''(t_{k,0})
= 2^{n-1} \gamma''(h^{-1}(2^{-k}))
$$
and since $h(t)=t^2 \gamma''(t)$,
$$\gamma''(h^{-1}(2^{-k}))=
\frac{[h^{-1}(2^{-k})]^2
\gamma''(h^{-1}(2^{-k}))}
{[h^{-1}(2^{-k})]^2}
= \frac{2^{-k}}{[h^{-1}(2^{-k})]^2} $$
Thus if $|\xi_2|\ge \eps|\xi_1|$ (i.e.$ |\xi_2|\approx 2^k$) we get by van der Corput's Lemma
$$|m_{k,n}(\xi)| \le C|\xi_2|^{-1/2}|\gamma''(t_{k,n-1})|^{-1/2}
\le C_\eps |\xi|^{-1/2} 2^{k/2} 2^{-\frac{n-1}{2}} h^{-1}(2^{-k})
$$
which is at most  $C_\eps'   2^{-n/2} \gamma^{-1}(2^{-k})$,
by \eqref{wk-tk0} .

If  $|\xi_2|\le \eps|\xi_1|$ and $\eps$ is sufficiently small
then the derivative of the phase is $|\xi_1+\xi_2\gamma'(t)|\ge |\xi_1|\approx 2^k$ and we obtain $|m_{k,n}(\xi)| \le C2^{-k}.$ Now $$2^{-k}=
h(t_{k,0})= 2^{-k/2}2^{\frac{1-n}{2}} \sqrt{h(t_{k,n-1})} \lc
\gamma^{-1}(2^{-k}) 2^{-n/2}$$
since $h$ is bounded and $\gamma(t) \lc t^2$. We have now proved
the estimate
\begin{equation} |m_{k,n}(\xi)\beta(2^{-k}|\xi|)| \lc   2^{-n/2} \gamma^{-1}(2^{-k})
\end{equation}
for $n>0$.
For $n=0$ we need to estimate
$$m_{k,0}(\xi)=
 \int_{0}^{t_{k,0}} \exp(i(\xi_1 t+\xi_2 \gamma(t))) dt
$$
and we just use the trivial bound
$$ |m_{k,0}(\xi)|\le t_{k,0}=h^{-1}(2^{-k})
\le \gamma^{-1}(2^{-k}).
$$
These estimates imply the  asserted $L^2$ bound  by Plancherel's theorem.
\end{proof}

\begin{proof}[Proof of Proposition \ref{AknLp}]
It suffices to show
\begin{equation}\label{vv}
\Big\|\Big(\sum_k | w_k A_{k,n}  f_k|^2\Big)^{1/2}\Big\|_p
\le C_{\eps,p} 2^{n(b+\eps)}
\Big\|\Big(\sum_k |f_k|^2\Big)^{1/2}\Big\|_p
\end{equation}
for $1<p<\infty$ since we can apply it with $f_k=L_k f$ where
$L_k$
 is a suitable Littlewood-Paley type localization operator  with
$P_k=P_kL_k$. By a duality argument it suffices to prove \eqref{vv} for
$1<p\le 2$.

Consider the maximal function
\Be\label{maxalong}\fM f(x)= \sup_{0<r\le 1} \frac{1}{r}\int_0^r |f(x_1-t, x_2-\gamma(t))| dt.\Ee
By our assumption $\gamma'''(t)\ge 0$ and $\gamma'(0)=\gamma''(0)=0$, the convex function
$u=\gamma'$ satisfies the doubling
condition $u(2t)\ge 2u(t)$ for $0\le t\le 1/2$. Thus by
\cite{eight}
the maximal operator $\fM$ is bounded on
$L^p(\bbR^2)$ for
$1<p<\infty$.
We also have  the vector-valued version
\begin{equation}\label{vvM}
\big\|\big\{\fM f_k\big\}_{k\in \bbZ} \big\|_{L^p(\ell^q)}
\lc
\big\|\big\{ f_k\big\}_{k\in \bbZ} \big\|_{L^p(\ell^q)}, \quad 1<p<\infty,\quad p\le q\le\infty\,.
\end{equation}
Indeed for $p=q$ this follows from the $L^p$ inequality  for $\fM$ and interchanging summation and integration. For $q=\infty$ it follows from
using $\sup_k \fM f_k= \fM (\sup_k|f_k|)$ and applying the $L^p$ inequality  for $\fM$. For $p\le q\le \infty$ we use a standard linearization and complex interpolation argument. In particular if $1<p\le 2$, \eqref{vvM} holds for $q=2$.

Now note the trivial majorization
$$|A_{k,n} f_k(x)| \le t_{k,n} \fM[P_kf_k] (x).$$
We shall show
that given $\eps>0$ in the statement of the proposition we have
\Be\label{wktkn}
\sup_{\substack {k,n :\\ n\le  N_k}} 2^{-n (b+\eps)}w_k t_{k,n} \le C(\eps) \,.
\Ee

Using   \eqref{wktkn} we may estimate
\begin{align*}
\Big\|\Big(\sum_k | w_k A_{k,n}  f_k|^2\Big)^{1/2}\Big\|_p
&\le C(\eps)
2^{n(b+\eps)}
\Big\|\Big(\sum_k | \fM  P_kf_k|^2\Big)^{1/2}\Big\|_p\\
&\le C'(\eps) 2^{n(b+\eps)} \Big\|\Big(\sum_k | f_k|^2\Big)^{1/2}\Big\|_p,
\end{align*}
by  \eqref{vvM}, and vector-valued singular integral estimates for the operator
$\{f_k\}\mapsto \{P_kf_k\}$  on $L^p(\ell^2)$.

 It remains to show  \eqref{wktkn}.
By \eqref{wk-tk0} the required bound holds for $n=0$.
Moreover
$$w_k t_{k,n}
\,\le\,
\frac{{\gamma''}^{-1}(\rho_{k,n})} {h^{-1}(2^{-k})}
\,=\,
\frac{{\gamma''}^{-1}(\rho_{k,n})}
{{\gamma''}^{-1}(\rho_{k,0})} .
$$
Consequently  it suffices to show that
$$\sup_{\substack {k,n :\\1\le  n\le  N_k}} 2^{-n (b+\eps)}
\frac{{\gamma''}^{-1}(\rho_{k,n})}
{{\gamma''}^{-1}(\rho_{k,0})}
\le C(\eps) \, ,
$$
or, equivalently
\Be\label{logs}-b \ln 2 + \frac 1n  \ln\big( \frac{{\gamma''}^{-1}(\rho_{k,n})}
{{\gamma''}^{-1}(\rho_{k,0})} \big ) \le \eps \ln (2) + \frac{\ln (C(\eps))}{n},
\Ee
for any $(k,n)$ with $k>k_\circ$ and $1\le n\le N_k$.
Note that $n^{-1} \ln(\rho_{k,n}/\rho_{k,0})=1$ for $n<N_k$ and
$1\le n^{-1} \ln(\rho_{k,n}/\rho_{k,0})\le \tfrac{n+1}{n}$ for  $n=N_k$.
The  left hand side of \eqref{logs} is then equal to
$$\frac{1}{n} \int_{\rho_{k,0}}^{\rho_{k,n}} \tau_b(s) \frac{ds}{s} + E_{k,n}
$$ where $$\tau_b(s) =
\frac{s}{\gamma'''({\gamma''}^{-1}(s)){\gamma''}^{-1}(s) } -b\, ,
$$
$E_{k,n} =0$ for $n<N_k$ and   $|E_{k,n}|\le n^{-1}$ for $n=N_k$.
Thus it suffices to show that  there is $\cN(\eps)>0$ so that
\begin{equation}\label{epsilonest}
\sup_{k: N_k>n}\frac{1}{n} \int_{\rho_{k,0}}^{\rho_{k,n}}| \tau_b(s)| \frac{ds}{s}
\le \eps \ln (2)\text{ for } n>\cN(\eps)\,.
\end{equation}
Now ${\gamma''}^{-1}(s)\to 0$ as $s\to 0$
and, since
$b=\lim_{t\to 0} \frac{\gamma''(t)}{t\gamma'''(t)}$, we have
 therefore $\lim_{s\to 0}  \tau_b(s)= 0$.
Choose $\delta(\eps)\in (0,\frac 12 \gamma''(1)) $ so that $|\tau_b(s)|<\eps/4$ whenever $s<\delta(\eps)$. Because of the assumed behavior of
$\frac {\gamma''(t)}{t\gamma'''(t)}$ near $0$ we see that
$M=\sup_{0< t\le 1} \frac {\gamma''(t)}{t\gamma'''(t)}<\infty$. Let
$\cN(\eps) = 4M\eps^{-1} \ln (\tfrac {\gamma''(1)}{\delta(\eps)})$.

Let  $I_{k,n}=[\rho_{k,0},\rho_{k,n}]\cap (0,\delta(\eps)]$
and $J_{k,n}=[\rho_{k,0}, \rho_{k,n}]\cap [\delta(\eps),\gamma''(1)]$; one of
 these  intervals may be empty.
Now $\tau_b (s)\le \eps/4$ on $I_{k,n}$  and $\rho_{k,n}/\rho_{k,0}< 2^{n+1}$.
Therefore for all $n\le N_k$
$$\frac{1}{n} \int_{I_{k,n}} |\tau_b(s)| \frac{ds}{s}\le \frac{\eps}{4n}
 \int_{\rho_{k,0}}^{\rho_{k,n}}  \frac{ds}{s} \le \frac{n+1}{n} \frac {\eps\ln 2}{4}\le \frac{\eps \ln (2)}{2}.$$
On $J_{k,n}$ we use the estimate  $|\tau_b(s)| \le 2M$ and obtain
for $\cN(\eps)\le n\le N_k$
$$\frac {1}{n} \int_{J_{k,n}} |\tau_b(s)| \frac{ds}{s}
\le \frac {1}{n} \int_{\delta(\eps)}^{\gamma''(1)} 2M  \frac{ds}{s}
= \frac{2M \ln (\tfrac{\gamma''(1)}{\delta(\eps)})}{\cN(\eps)}
\le\frac {\eps}{2}.$$
We combine these two estimates and obtain \eqref{epsilonest}. Thus the proposition
is proved.
\end{proof}

\section{Another open problem}
\label{hyp}

We recall another open problem concerning
a lacunary maximal operator generated by dilates of a Marcinkiewicz multiplier in two dimensions.
Let $\eta_0$ be a Schwartz function on the real line with
$\eta_0(0)\neq 0.$  Let $$m_k(\xi_1,\xi_2)= \eta_0(2^{2k}\xi_1\xi_2)$$
(the dilates of the so-called hyperbolic cross multiplier) and define
$$\bbM  f(x)= \sup_{k}|\cF^{-1}[m_k\widehat f](x)|.$$

\noi{\it Problem.} Is $\bbM$ bounded on $L^p(\bbR^2)$, for some $p\in (1,\infty)$?

The  problem is closely related to one formulated in \cite{elkohen} and
\cite{dt}, on the
 pointwise convergence for the
``hyperbolic'' Riesz means $R_{\la,t}f$. These are defined by  $\widehat {R_{\la,t} f}(\xi)= (1-t^{-2}\xi_1^2\xi_2^2)_+^\la \widehat f$ and they were
studied in \cite{elkohen} and \cite{carbery}.
 A positive answer to our question is known to  imply
positive $L^p(\bbR^2)$
boundedness results for the maximal function
$M_\la f(x)=\sup_{t>0} |R_{\la,t} f(x)|$, for suitable $\la$.
A negative answer  would prove that $M_\la$ is unbounded on all $L^p(\bbR^2)$, for all $\la$.

Note that the multipliers $m_k$ satisfy the hypotheses
of the Marcinkiewicz  multiplier theorem in $\bbR^2$.
As proved  in \cite{cghs}
$L^p$ boundedness for lacunary maximal functions generated by
Mikhlin-H\"ormander (or Marcinkiewicz)  multipliers
 fails generically, with respect to the topology in  some natural symbol
spaces.
This however does not settle our  question above.  More results and
open problems on  minimal decay assumptions for the boundedness of
such maximal
operators can be found in \cite{ghs}.

One can also ask for bounds on the $L^p$ operator norm
of the maximal operator
$\bbM_N  f(x)= \sup_{|k|\le N}|\cF^{-1}[m_k\widehat f](x)|.$
Petr Honz\'\i k \cite{honzik} noted that
 one can improve the trivial upper bound $C_pN^{1/p}$
by combining   the   good $\la$ inequalities in  \cite{pipher} with the reasoning in  \cite{ghs}; this yields the bound $C_p \log(N)$ (at least for a discrete analogue).


\begin{thebibliography}{10}

\bibitem{bnw} J. Bruna, A. Nagel, S.  Wainger,  Convex hypersurfaces and Fourier transforms.  Ann. of Math. (2)  127  (1988),  no. 2, 333--365.

\bibitem{carbery}
A. Carbery, \emph{A note on the ``hyperbolic'' Bochner-Riesz means.}  Proc. Amer. Math. Soc.  92  (1984),  no. 3, 397--400.


\bibitem{eight} H. Carlsson, M.  Christ, A.  C\'ordoba,
J.  Duoandikoetxea,
J.L.  Rubio de Francia, J.  Vance, S.  Wainger, D.  Weinberg,
\emph{$L^p$ estimates for maximal functions and Hilbert transforms along flat convex curves in $R^2$.}  Bull. Amer. Math. Soc. (N.S.)  14  (1986),  no. 2, 263--267.

\bibitem{crf} S.Y.A.  Chang, R. Fefferman, \emph{A continuous version of duality of $H^1$ and $BMO$ on the bidisc.} Annals of Math. 112
 (1980), 179--201.


\bibitem{christ}M. Christ, \emph{Weak type (1,1) bounds for rough
operators.}  Annals of Math. 128 (1988), 19--42.

\bibitem{christ2}\bysame, \emph{Failure of an endpoint estimate for integrals along curves.}  Fourier analysis and partial differential equations
(Miraflores de la Sierra, 1992),  163--168, Stud. Adv. Math., CRC,
Boca Raton, FL, 1995.

\bibitem{cghs} M. Christ, L. Grafakos, P. Honz\'\i k, A. Seeger,
\emph{Maximal functions associated with Fourier multipliers of Mikhlin-H\"ormander type.}  Math. Z.  249  (2005),  no. 1, 223--240.

\bibitem{dt} H. Dappa, W.  Trebels, \emph{On maximal functions generated by Fourier multipliers.}  Ark. Mat.  23  (1985),  no. 2, 241--259.



\bibitem{duordf}
J. Duoandikoetxea, J. Rubio de Francia, \emph{Maximal and singular integral operators via Fourier transform estimates.}
Invent. Math. 84 (1986), no. 3, 541--561.

\bibitem{elkohen}
A. El Kohen, \emph{On the hyperbolic Riesz means.}
Proc. Amer. Math. Soc. 89 (1983), no. 1, 113--116.




\bibitem{fs} C. Fefferman, E.M.  Stein,
 \emph{Some maximal inequalities.}
Amer. J. Math.  93  (1971),  107--115.




\bibitem{ghs} L. Grafakos, P. Honz\'\i k, A. Seeger \emph{On maximal functions for Mikhlin-H\"ormander multipliers.}  Adv. Math.  204  (2006),  no. 2, 363--378.

\bibitem{heo1} Y. Heo, \emph{An endpoint estimate for some maximal operators associated to submanifolds of low codimension.} Pacific J. Math. 201 (2001), no. 2, 323--338.

\bibitem{heo2} \bysame,
\emph{Weak type estimates for some maximal operators on Hardy spaces.} Math. Nachr. 280 (2007), no. 3, 281--289.



\bibitem{heo3} \bysame, \emph{Endpoint estimates for some maximal operators associated to the circular conical surface.}  J. Math. Anal. Appl.  351  (2009),  no. 1, 152--162.




\bibitem{hns} Y. Heo, F. Nazarov, A. Seeger,
\emph{ Radial Fourier multipliers
  in high dimensions.} To appear in Acta Math.

\bibitem{honzik}
P. Honz\'\i k,
Personal communication.


\bibitem{oberlin} D. M. Oberlin, \emph{An endpoint estimate for some maximal operators.} Rev. Mat. Iberoamerica 12, No. 3, 641--652 (1996).



\bibitem{peetre} J. Peetre, \emph{On spaces of Triebel-Lizorkin type.}
Ark. Mat.  13  (1975), 123--130.

\bibitem{pipher}
J. Pipher, \emph{Bounded double square functions.}  Ann. Inst. Fourier (Grenoble)  36  (1986),  no. 2, 69--82.

\bibitem{pod1} A.N. Podkorytov, \emph{On the asymptotics of the Fourier transform on a convex curve.} (Russian)  Vestnik Leningrad. Univ. Mat. Mekh. Astronom.  1991, , vyp. 2, 50--57, 125;  translation in  Vestnik Leningrad Univ. Math.  24  (1991),  no. 2, 57--65.


\bibitem{pod2} \bysame,
\emph{What is it possible to say about an asymptotic of the Fourier transform of the characteristic function of a two-dimensional convex body with nonsmooth boundary?}  Fourier analysis and convexity,  209--215, Appl. Numer. Harmon. Anal., Birkhäuser Boston, Boston, MA, 2004.

\bibitem{pr-ro-se} M. Pramanik, K. Rogers and and A. Seeger,
\emph{A Calder\'on-Zygmund estimate with applications
to generalized Radon transforms and Fourier integral operators.}
Studia Mathematica, to appear.

\bibitem{pr-se} M. Pramanik and A. Seeger,
\emph{$L\sp p$ regularity of averages over curves and bounds for
  associated maximal operators.}  Amer. J. Math.  129  (2007),  no. 1,
61--103.



\bibitem{se-stud} A. Seeger, \emph{Remarks on singular convolution operators.}  Studia Math.  97  (1990), 91--114.

\bibitem{se-duke} \bysame, \emph{Degenerate Fourier integral operators in the plane.}  Duke Math. J.  71  (1993),  no. 3, 685--745.

\bibitem{setao} A. Seeger and T. Tao,
\emph{Sharp Lorentz space estimates for rough operators,}
Math. Ann. 320 (2001), 381-415.

\bibitem{stw1}
A. Seeger, T. Tao and J. Wright,
\emph{Endpoint mapping properties of spherical maximal operators.}
Journal de l'Institut Math\'ematiques de Jussieu, 2 (2003), 109-144.


\bibitem{stw2}
\bysame,
\emph{Singular maximal functions and Radon transforms near $L^1$.}
Amer. J. Math.  126  (2004),  no. 3, 607--647.
See also

http://www.math.wisc.edu/$\sim$seeger/papers/corr-stw-amj.pdf
for a correction.




\bibitem{SoSt}  C.D. Sogge, E. M. Stein,
\emph{Averages over hypersurfaces. Smoothness of generalized Radon transforms.}  J. Analyse Math.  54  (1990), 165--188.

\bibitem{steinbook}
E.M. Stein, \emph{Harmonic analysis: real-variable methods, orthogonality, and oscillatory integrals.} Princeton Mathematical Series, 43. Monographs in Harmonic Analysis, III. Princeton University Press, Princeton, NJ, 1993.

\bibitem{StW} E.M. Stein, S.  Wainger,
\emph{Problems in harmonic analysis related to curvature.}  Bull. Amer. Math. Soc.  84  (1978), no. 6, 1239--1295.

\bibitem{taocyl}
T. Tao, \emph{The $H^1 \to L^{1,\infty}$ boundedness
of the cylindrical maximal function}, unpublished  notes 2002.


\bibitem{wolff}
T. Wolff, \emph{Local smoothing type estimates on $L^p$ for large $p$.}
Geom. Funct. Anal.  10  (2000),  no. 5, 1237--1288.


\bibitem{yang} C.W. Yang,
\emph{$L^p$ regularity of averaging operators with higher fold singularities.}  Proc. Amer. Math. Soc.  131  (2003),  no. 2, 455--465.




\end{thebibliography}
\end{document}